\documentclass[11pt]{amsart}
\usepackage[margin=1in]{geometry}
\usepackage{amsmath, amsfonts, amssymb, amsthm}
\usepackage{amsrefs}
\usepackage{mathrsfs}
\usepackage{enumitem} 
\usepackage{hyperref}
\usepackage{color}
\usepackage{dsfont}
\usepackage{csquotes}
\usepackage{epigraph}
\usepackage{comment}
\usepackage{mlmodern}
\usepackage[T1]{fontenc}

\usepackage{color}
\definecolor{darkblue}{rgb}{0.0,0.0,0.3}
\hypersetup{
    colorlinks=false,
    linkcolor=blue,
    urlcolor=darkblue,
    }

\newtheorem{theorem}{Theorem}[section]

\newtheorem{conjecture}[theorem]{Conjecture}

\newtheorem{lemma}[theorem]{Lemma}

\theoremstyle{definition}

\newtheorem{definition}[theorem]{Definition}

\newtheorem{notation}[theorem]{Notation}

\theoremstyle{remark}
\newtheorem{remark}[theorem]{Remark}

\numberwithin{equation}{section}

\newcommand{\bR}{\mathbb{R}}

\newcommand{\cD}{\mathcal{D}}

\newcommand{\Ampere}{Amp\`{e}re}

\newcommand\norm[1]{\left\lVert#1\right\rVert}

\DeclareMathOperator{\diam}{diam}

\DeclareMathOperator{\const}{const}

\title[The prescribed curvature equations in Minkowski space]{The Dirichlet problem for the prescribed curvature equations in Minkowski space}
\author{Bin Wang}
\address[]{Department of Mathematics, the Chinese University of Hong Kong, Shatin, New Territories, The Hong Kong Special Administrative Region of the People's Republic of China.} 
\email{bwang@math.cuhk.edu.hk; bwangmath@outlook.com}

\begin{document}

\setcounter{tocdepth}{1} 
\begin{abstract}
We study the Dirichlet problem for functions whose graphs are spacelike hypersurfaces with prescribed curvature in the Minkowski space and we obtain some new interior second order estimates for admissible solutions to the corresponding fully nonlinear elliptic partial differential equations.
\end{abstract}
\maketitle
\tableofcontents

\section{Introduction}
The Dirichlet problem for the prescribed curvature equations in the Minkowski space is a physically motivated problem due to its applications in the theory of relativity. However, the problem has not been fully explored in the literature compared to its Euclidean counterparts because the geometric nature of the ambient space causes substantial difficulties for deriving a priori estimates of admissible solutions. Hence, major advancements for this problem have stopped since the work of P.~Bayard \cite{Bayard} and J.~Urbas \cite{Urbas} in 2003, and only until very recent years C.~Ren and Z.~Wang \cite{Ren-Wang-1, Ren-Wang-2, Wang} have made some important progress. In this note, we continue the investigation from the limited literature and obtain some new interior curvature estimates which also improve those already remarkable ones due to J.~Urbas \cite{Urbas} and Z.~Wang \cite{Wang}. A more detailed literature review will soon be given below and our main results are stated in Theorem \ref{k=2} and Theorem \ref{semi-convex}.

\begin{remark}
For literature on the Euclidean case, the reader may be referred to the work of Guan-Ren-Wang \cite{Guan-Ren-Wang} and Sheng-Urbas-Wang \cite{Sheng-Urbas-Wang}; see also the very delicate work \cite{Guan-Spruck} of Guan-Spruck in hyperbolic space.
\end{remark}

Let $\sigma_k: \bR^n \to \bR$ denote the $k$-th elementary symmetric polynomial, which is defined as
\[\sigma_k(\kappa_1,\ldots,\kappa_n)=\sum_{1 \leq i_1<i_2<\cdots<i_k \leq n} \kappa_{i_1}\cdots \kappa_{i_k}.\] We consider a fully nonlinear elliptic equation of the form
\[F[u]:=f(\kappa_1(x),\ldots,\kappa_n(x))=\psi(x,u (x)), \quad x \in \Omega,\] in a smooth bounded domain $\Omega \subseteq \bR^n$, where $f$ is a smooth symmetric function of $n$ variables given by $\sigma_k$, $\psi$ is a prescribed positive function and $\kappa=(\kappa_1,\ldots,\kappa_n)$ denotes the principal curvatures of the graph of $u$ over $\Omega$. 

The study of this class of fully nonlinear elliptic equations were initiated by Caffarelli-Nirenberg-Spruck \cite{CNS-3,CNS-4,CNS-5} and Ivochkina \cite{Ivochkina-1, Ivochkina-2, Ivochkina-3}, and further developed by Trudinger-Wang \cite{TW-1,TW-2,TW-3}. Indeed, the equation operator $\sigma_k$ includes a large class of notable examples. For a smooth hypersurface $\Sigma$ with principal curvatures $\kappa[\Sigma]=(\kappa_1,\ldots,\kappa_n)$,  the quantity $\sigma_k(\kappa[\Sigma])$ will be called the $k$-th mean curvature (or just $k$-curvature) of $\Sigma$. In particular, $\sigma_1(\kappa[\Sigma])=\sum_{i=1}^{n} \kappa_i$ is the usual mean curvature, $\sigma_2(\kappa[\Sigma])=\sum_{i<j}\kappa_i\kappa_j$ is the scalar curvature, and $\sigma_n(\kappa[\Sigma])=\kappa_1\cdots\kappa_n$ is the Gauss curvature. The other values of $k$ also pertain to some important geometric problems in the sense that they can be reduced to solving some particular $\sigma_k$ type equations; see e.g. \cite{Guan-Guan} and \cite{Guan-Li-Li, Yang}. 

The aim of this note is to seek a $k$-admissible function $u$ such that its graph $\Sigma=(x,u(x))$ over $\overline{\Omega}$ is a spacelike hypersurface in the Minkowski space and solves the following Dirichlet problem
\begin{gather} \label{Dirichlet problem}
\begin{split}
       \sigma_k[u]&=\psi(x,u)\quad  \text{in $\Omega$},\\
                u&=\varphi \ \quad\quad\quad     \text{on $\partial\Omega$},
\end{split}
\end{gather} where the meaning of being spacelike and $k$-admissible are specified in section \ref{Preliminaries}. 

The problem was first studied by Robert Bartnik and Leon Simon in their influential paper \cite{Bartnik-Simon} for the case of prescribing mean curvature i.e. when $k=1$ in \eqref{Dirichlet problem}; see also the extension by Gerhardt \cite{Gerhardt-1} to non-flat spacetimes. The motivations mainly came from Einstein's theory of relativity, as first emphasized in the fundamental paper of Lichnerowicz \cite{L}. Roughly speaking, spacelike hypersurfaces of constant mean curvature in the Minkowski space are important because they provide Riemannian submanifolds with properties which reflect those of the spacetime. In particular, they played a role in the initial proof \cite{Schoen-Yau} of the positive mass conjecture by Schoen and Yau. On the other hand, the study of such hypersurfaces was already posed by Calabi \cite{Calabi} at an earlier time for the quest of understanding its Bernstein type property; see the resolution of this problem by Cheng-Yau \cite{Cheng-Yau} and references citing their papers.

Later, \eqref{Dirichlet problem} was solved by Delano\"{e} \cite{Delanoe} for $k=n$ i.e. the prescribed Gauss curvature equation; see also the work of Guan \cite{Guan} in which the result was improved under a subsolution condition. The next interesting case would naturally be to solve the prescribed scalar curvature equation i.e. when $k=2$ in \eqref{Dirichlet problem}. Bayard \cite{Bayard} was the first to tackle the problem and proved the solvability in dimensions three and four. Soon after that, Urbas \cite{Urbas} extended this result to all dimensions, however, Urbas' proof relied crucially on the additional assumption that both the boundary data $\varphi$ and the domain are uniformly convex in the sense that $D^2\varphi \geq c_0I$ in $\overline{\Omega}$ for some uniform constant $c_0>0$ and the principal curvatures of $\partial \Omega$ are bounded from below by a positive uniform constant. Our first main result is the removal of this assumption and the following improved existence theorem is obtained.

\begin{theorem} \label{k=2}
Let $\Omega \subseteq \bR^n$ be a convex $1$-admissible bounded domain with a smooth boundary. Suppose $\psi \in C^{\infty}(\overline{\Omega} \times \bR)$ is a positive function satisfying $\psi_u \geq 0$ and $\varphi \in C^4(\overline{\Omega})$ is spacelike. Assume the existence of a $2$-admissible subsolution $\underline{u}$ such that
\begin{align*}
\sigma_2[\underline{u}]&\geq \psi(x,\underline{u})\quad  \text{in $\Omega$},\\
                \underline{u}&=\varphi \ \quad\quad\quad     \text{on $\partial\Omega$}.
\end{align*} Then there exists a unique $2$-admissible solution $u$ to 
\begin{gather} \label{Dirichlet problem k=2}
\begin{split}
\sigma_2[u]&=\psi(x,u) \quad  \text{in $\Omega$},\\
                u&=\varphi \ \quad\quad\quad     \text{on $\partial\Omega$},
\end{split}
\end{gather}
belonging to $C^{3,\alpha}(\overline{\Omega})$ for any $\alpha \in (0,1)$.
\end{theorem}

The next goal is to continue the investigation for the remaining cases $3 \leq k \leq n-1$ and we shall proceed by the standard continuity method along with the regularity theorem due to Evans-Krylov \cite{Evans,Krylov}; this requires us to establish a priori estimates for admissible solutions up to the second order. In fact, the $C^0$ estimate follows directly from the comparison principle and the $C^1$ estimate has been successfully obtained by Bayard \cite[Proposition 3.1]{Bayard} for all $k$. For the second order derivatives on the boundary $\partial \Omega$, Bayard \cite[Section 4.2]{Bayard} proved the bound for all $k$ if $\varphi=\const$ and for a general boundary data, the condition $k=2$ had to be imposed. Although the boundary $C^2$ estimate has not been obtained in the most general case, it is good enough for the moment. The real issue is the second order estimate inside $\Omega$ which was not known for $3 \leq k \leq n-1$ even in some very simple cases.

The major difficulties come from two aspects. First, the operator $\sigma_k$ is a lot more structurally complicated when $3 \leq k \leq n-1$. As a comparison, when $k=1$ the equation is quasilinear and when $k=n$ the equation is of Monge-\Ampere\ type, both of which are extensively studied in the literature and many techniques could be adapted to our setting. Even when $k=2$, there are structural advantages that were utilized by Bayard \cite{Bayard} and Urbas \cite{Urbas} but failed to hold for $k \geq 3$. The second major obstacle occurs in the interchanging formula \eqref{interchanging formula 1} which gives rise to a negative curvature term. In contrast, the formula yields a positive curvature term for the Euclidean case which is rather crucial as demonstrated in \cite{Sheng-Urbas-Wang}.

In two recent papers \cite{Ren-Wang-1, Ren-Wang-2}, C.~Ren and Z.~Wang proved two powerful concavity inequalities for the $\sigma_{n-1}$ and $\sigma_{n-2}$ operators by exploiting their algebraic structures in depth, and then they were able to overcome the difficulties and derive the second order estimates for $k=n-1$ and $k=n-2$. However, it may not be feasible to generalize their method for other $k$'s. In \cite{Huang}, Huang assumed $\psi=\psi(x,u,Du)$ has a special dependence on the gradient terms so that an extra positive curvature term could be extracted from the twice differentiation of the equation. Huang's method is inspiring but the assumptions are not applicable when e.g. $\psi=\psi(x,u)$ does not contain the gradient terms at all.

Now the ultimate goal is reduced to deriving second order estimates in $\Omega$ for $3 \leq k \leq n-3$ and general $\psi$. Using the novel techniques that were developed by Guan-Ren-Wang \cite{Guan-Ren-Wang}, Z.~Wang \cite{Wang} has achieved the goal for $k$-admissible solutions whose graphs are $(k+1)$-convex. Our second main result establishes the estimates for $k$-admissible solutions whose graphs are semi-convex, which may attract more attention than Theorem \ref{k=2}.

\begin{theorem}\label{semi-convex}
Let $2 \leq k \leq n-1$ and let $\Omega \subseteq \bR^n$ be a convex $(k-1)$-admissible domain with a smooth boundary. Assume $\psi(x,u,Du) \in C^{2}(\overline{\Omega} \times \bR \times \bR^n)$ is positive and $\varphi \in C^2(\overline{\Omega})$ is spacelike. Suppose $u \in C^{4}(\Omega) \cap C^2(\overline{\Omega})$ is a $k$-admissible solution to 
\begin{align*}
       \sigma_k[u]&=\psi(x,u, Du)\quad  \text{in $\Omega$},\\
                u&=\varphi \quad\quad\quad\quad\quad     \text{on $\partial\Omega$},
\end{align*}
and the graph $\Sigma$ of $u$ is semi-convex i.e. the principal curvatures $\kappa=(\kappa_1,\ldots,\kappa_n)$ of $\Sigma$ satisfy
\begin{equation}
\kappa_{i}(x) \geq -K \quad \text{for all $x \in \Omega$ and all $1 \leq i \leq n$} \label{semi-convexity assumption}
\end{equation}
for some $K>0$. Then
\[\max_{\Omega} \kappa_{\max}(x) \leq C\left(1+\max_{\partial \Omega} \kappa_{\max}(x)\right)\] for some $C>0$ depending on $n,k,K, \norm{u}_{C^1(\overline{\Omega})}, \norm{\psi}_{C^2(\cD)}$ and $\norm{\varphi}_{C^1(\overline{\Omega})}$, where
\[\cD:=\overline{\Omega} \times [\inf_{\Omega} u, \sup_{\Omega} u] \times \bR^n.\]
\end{theorem}
\begin{remark}
The new contribution here is that our estimate improves that of Z.~Wang \cite{Wang} in the sense that semi-convexity is weaker than $(k+1)$-convexity; for a proof of this fact, the reader is referred to \cite[Lemma 7]{Li-Ren-Wang} or \cite[Lemma 2.13]{Bin-1}.
\end{remark}

\begin{remark}
If one could obtain the curvature estimate without the additional semi-convexity assumption \eqref{semi-convexity assumption}, then we would be able to conclude the solvability of the Dirichlet problem \eqref{Dirichlet problem} for all $k$, at least for a constant boundary data. 
\end{remark}

In \cite{Ren-Wang-2}, C.~Ren and Z.~Wang had conjectured that their concavity inequality should hold for all $k$ with $2k>n$. Hence, the desired curvature estimate is expected to hold for all $k$ with $2k>n$. Although a verification for their conjecture is still absent, we may conjecture the following.

\begin{conjecture} 
Let $2 \leq k \leq n-1$ satisfy $2k>n$ and let $\Omega \subseteq \bR^n$ be a convex $(k-1)$-admissible bounded domain with a smooth boundary. Suppose $\psi(x,u,Du)$ is a smooth positive function satisfying $\psi_u \geq 0$ and $\varphi \in C^4(\overline{\Omega})$ is spacelike. Assume the existence of a $k$-admissible subsolution $\underline{u}$ such that
\begin{align*}
\sigma_k[\underline{u}]&\geq \psi(x,\underline{u}, D\underline{u})\quad  \text{in $\Omega$},\\
                \underline{u}&=\varphi \quad\quad\quad\quad\quad     \text{on $\partial\Omega$}.
\end{align*} Then there exists a unique $k$-admissible solution $u$ to 
\begin{align*}
\sigma_k[u]&= \psi(x,u, Du)\quad  \text{in $\Omega$},\\
                u&=\varphi \quad\quad\quad\quad\quad     \text{on $\partial\Omega$},
\end{align*}
belonging to $C^{3,\alpha}(\overline{\Omega})$ for any $\alpha \in (0,1)$.
\end{conjecture}

To conclude the introduction, we mention some related research. In \cite{Schnurer}, Schn\"{u}rer considered the Dirichlet problem in a general Lorentzian product manifold for a class of curvature functions which excludes the operator $\sigma_k(\kappa)$ concerned here. On the other hand, Gerhardt \cite{Gerhardt-2, Gerhardt-3, Gerhardt-4} studied closed hypersurfaces of prescribed curvature in Lorentzian manifolds as well. Indeed, much of our proof remains valid in a more generic setting, however, since the problem has not been fully solved even in some very simple cases (say, the case when $\psi=\const$), it might not be worthwhile to create more complications at this moment. For literature on entire spacelike hypersurfaces, the reader may be referred to the series of work \cite{Math-Ann, JFA, CVPDE, JGA, Analysis-PDE, Wang-Xiao-1, Wang-Xiao-2, Bayard-2009, Bayard-Delanoe, Bayard-Schnurer, Bayard-2006, Bayard-Seppi, Bayard-2023} by P.~Bayard, Ph.~Delano\"{e}, C.~Ren, A.~Seppi, O.~C.~Schn\"{u}rer, Z.~Wang, L.~Xiao. We also wish to call attention to recent papers \cite{Guo-Jiao, Guo-Jiao-2} of Guo, H.~Jiao, and Y.~Jiao, in which they studied the same Dirichlet problem for different classes of fully nonlinear equations. Another recent paper \cite{rigidity} studies rigidity theorems for spacelike hypersurfaces.

The rest of this note is organized as follows. In section \ref{Preliminaries}, we review preliminary concepts and fix notations. In section \ref{gradient estimates}, we derive global gradient estimates for admissible solutions which extend the ones obtained by Bayard \cite{Bayard} in the sense that our estimates hold for a right-hand side $\psi=\psi(x,u)$ that may depend on $u$ and we have replaced Bayard's assumption of strict convexity on both $\Omega$ and $\varphi$ by the subsolution condition. In section \ref{k=2 proof}, we prove curvature estimates for admissible solutions to the scalar curvature equation in all dimensions without assuming uniform convexity on $\Omega$ and $\varphi$, hence improving Urbas' result \cite{Urbas} and Theorem \ref{k=2} follows accordingly. In section \ref{curvature estimate}, we prove Theorem \ref{semi-convex} which extends the estimate due to Z.~Wang \cite{Wang}. It might be noteworthy that in both section \ref{k=2 proof} and section \ref{curvature estimate}, instead of ordinary curvature estimates, we actually have proved Pogorelov type interior curvature estimates assuming the boundary data is affine. Another thing to note is that our curvature estimates hold for a general right-hand side $\psi=\psi(x,u,Du)$ that may even depend on $Du$.

\section{Preliminaries} \label{Preliminaries}
In this section, we review some basic concepts and fix notations which will be used throughout the subsequent sections but without directly quoting them every time.

Recall that the Minkowski space $\bR^{n,1}$ is the Euclidean space $\bR^{n+1}$ equipped with the metric
\[ds^2=dx_{1}^2+\cdots+dx_{n}^2-dx_{n+1}^2.\]  

\begin{definition}
A smooth hypersurface $\Sigma$ in $\bR^{n,1}$ is said to be $k$-convex if its principal curvatures 
\[\kappa[\Sigma] \in \Gamma_k=\{\kappa \in \bR^n: \sigma_{j}(\kappa)>0 \quad \forall\ 1 \leq j \leq k\}\] at every point; it is said to be spacelike if for every $p \in \Sigma$, the induced metric $\langle \cdot, \cdot \rangle_{T_p\Sigma}$ on the tangent space is positive definite.
\end{definition}

Let $\Sigma$ be the hypersurface in $\bR^{n,1}$ given as the graph of a smooth function $u: \overline{\Omega} \to \bR$. The induced metric and the second fundamental form of $\Sigma$ are then
\[g_{ij}=\delta_{ij}-u_iu_j, \quad h_{ij}=\frac{u_{ij}}{\sqrt{1-|Du|^2}}.\] 

\begin{remark}
The graph of $u$ over $\overline{\Omega}$ is a spacelike hypersurface in $\bR^{n,1}$ if and only if 
\begin{equation}
\sup_{\overline{\Omega}} |Du|<1. \label{spacelike condition}
\end{equation}
Thus, we may as well say a function $u$ is spacelike in $\overline{\Omega}$ if \eqref{spacelike condition} holds.
\end{remark}
The unit normal vector field to $\Sigma$ is 
\[\nu=\frac{(Du,1)}{\sqrt{1-|Du|^2}}.\] Note that we shall use $Du=(u_1,\ldots,u_n)$ and $D^2u=(u_{ij})$ to denote the ordinary gradient vector and the ordinary Hessian matrix. On the other hand, for a chosen local orthonormal frame $\{e_1,\ldots,e_n\}$ on $T\Sigma$, the symbol $\nabla$ will denote the induced Levi-Civita connection on $\Sigma$. For a smooth function $u$ on $\Sigma$, we set $\nabla_i u = \nabla_{e_i} u$ and $\nabla_{ij} u = \nabla^2 u(e_i,e_j)$. The norm of $\nabla u$ with respect to $g_{ij}$ is then
\[|\nabla u|=\sqrt{g^{ij}u_iu_j}=\frac{|Du|}{\sqrt{1-|Du|^2}},\] where
\[g^{ij}=\delta_{ij}+\frac{u_iu_j}{1-|Du|^2}\] is the inverse of $g_{ij}$.

We also recall the following fundamental formulas for hypersurfaces in $\bR^{n,1}$.
\begin{align}
\nabla_{ij} X&=h_{ij}\nu, \tag{Gauss formula}\\
\nabla_i\nu&=h_{ij}e_{j}, \tag{Weingarten formula}\\
\nabla_kh_{ij}&=\nabla_j h_{ik}, \tag{Codazzi equation}\\
R_{ijst}&=-(h_{is}h_{jt}-h_{it}h_{js}). \tag{Gauss equation}
\end{align}

For a symmetric matrix $A=(a_{ij})$ and an operator
\[F: \{\text{symmetric matrices}\} \to \bR,\] we define
\[F^{ij}=\frac{\partial F}{\partial a_{ij}}, \quad F^{ij,rs}=\frac{\partial^2 F}{\partial a_{ij} \partial a_{rs}}.\] When $F(A)=f(\lambda(A))$ depends only on the eigenvalues and when $A$ is diagonal, we have $F^{ij}=f_i\delta_{ij}$ where
\[f_i=\frac{\partial f}{\partial \lambda_i}.\] Moreover, we have
\[\sum_{i,j} F^{ij}a_{ij}=\sum_{i=1}^{n} f_i(\lambda(A))\lambda_i, \quad \sum_{i,j,k} F^{ij}a_{ik}a_{jk}=\sum_{i=1}^{n} f_i(\lambda(A))\lambda_{i}^2.\]

In this article, we are considering an equation of the form
\[F(A)=f(\lambda(A))=\psi(x,u,Du),\] where $f=\sigma_k$ and $A=(a_{ij})$ is given by
\[a_{ij}=\frac{1}{w}\gamma^{ik}u_{kl}\gamma^{lj}, \quad \gamma^{ik}=\delta_{ik}+\frac{u_iu_k}{w(1+w)},\] since the principal curvatures $\kappa=(\kappa_1,\ldots,\kappa_n)$ of $\Sigma$ are eigenvalues of the following matrix
\[\frac{1}{w}\left(I+\frac{Du \otimes Du}{w^2}\right)D^2u, \quad w=\sqrt{1-|Du|^2}.\]

To solve our equation, we need also to define the notion of admissible solutions.

\begin{definition}
A function $u \in C^{\infty}(\Omega) \cap C(\overline{\Omega})$ is said to be $k$-admissible if its graph $\Sigma=(x,u(x))$ over $\overline{\Omega}$ is a $k$-convex spacelike hypersurface in $\bR^{n,1}$. Equivalently, $u$ is $k$-admissible if \eqref{spacelike condition} holds and 
\[\sigma_j[u]>0 \quad \text{for all $x \in \overline{\Omega}$ and all $1 \leq j \leq k$}.\]
\end{definition}

\begin{remark}
The homogenized equation operator $\sigma_{k}^{1/k}[u]$ is concave with respect to the second derivatives for $k$-admissible solutions. This is a key condition to invoke the Evans-Krylov device in \cite{Evans, Krylov}.
\end{remark}

We shall always assume $\Omega$ is an admissible domain.

\begin{definition}
A smooth bounded domain $\Omega \subseteq \bR^n$ is said to be $(k-1)$-admissible if at least $k-1$ principal curvatures of $\partial \Omega$ (relative to the interior normal) are positive at each boundary point. 
\end{definition}

\begin{lemma}
Every affine spacelike boundary data $\varphi$ on $\partial \Omega$ has a $k$-admissible extension to $\overline{\Omega}$ if and only if the smooth bounded domain $\Omega$ is convex and $(k-1)$-admissible.
\end{lemma}
\begin{proof}
See \cite[Lemma 2.1]{Bayard}.
\end{proof}

\begin{remark}
Here $\Omega$ being convex means that its principal curvatures $\kappa_{1}^{b}(x_0), \ldots, \kappa_{n-1}^{b}(x_0)$ are all non-negative for $x_0 \in \partial \Omega$. Also, according to our definition, an $(n-1)$-admissible domain is just a strictly convex domain i.e. $\kappa_{1}^{b}(x_0), \ldots, \kappa_{n-1}^{b}(x_0)$ are all strictly positive. Similarly, an $n$-admissible function is just a strictly convex function that is spacelike.
\end{remark}

Finally, we state some commonly used properties of the $\sigma_k$ operator.

\begin{notation}
Observe that
\[\frac{\partial}{\partial \kappa_i}\sigma_k(\kappa)=\sigma_{k-1}(\kappa)\bigg|_{\kappa_i=0}=\sigma_{k-1}(\kappa_1,\ldots,\kappa_{i-1},0,\kappa_{i+1},\ldots,\kappa_n).\] Throughout the article, we will use $\sigma_{k-1}(\kappa|i)$ or $\sigma_{k}^{ii}(\kappa)$ interchangeably to denote the first order derivatives. The notations $\sigma_{k-2}(\kappa|ij)$ or $\sigma_{k}^{ii,jj}(\kappa)$ are defined in a similar way for second order derivatives. 
\end{notation}

\begin{lemma} \label{basic properties}
For all $1 \leq k \leq n$ and $\kappa \in \bR^n$, we have
\begin{align*}
\sigma_k(\kappa)&=\kappa_i\sigma_{k-1}(\kappa|i)+\sigma_{k}(\kappa|i),\\
\sum_{i=1}^{n}\kappa_i\sigma_{k-1}(\kappa|i)&=k\sigma_k(\kappa), \\
\sum_{i=1}^{n} \sigma_{k-1}(\kappa|i)&=(n-k+1)\sigma_{k-1}(\kappa).
\end{align*} Moreover, if $\kappa_1 \geq \cdots \geq \kappa_n$ and $\kappa \in \Gamma_k$, then 
\[\sigma_{k}^{11}(\kappa) \cdot \kappa_1 \geq \frac{k}{n} \sigma_k(\kappa).\]
\end{lemma}

\begin{lemma}[Maclaurin's inequality]
Let $2 \leq k \leq n$ and suppose $\kappa=(\kappa_1,\ldots,\kappa_n) \in \Gamma_k$. Denote by
\[H_{k}:=\binom{n}{k}^{-1}\sigma_k.\] Then we have
\[H_{k}^{1/k}(\kappa) \leq \cdots \leq H_{2}^{1/2} (\kappa) \leq H_1(\kappa).\]
\end{lemma}

\begin{lemma} \label{negative kappa}
Let $\kappa=(\kappa_1,\ldots,\kappa_n) \in \Gamma_k$. Suppose $\kappa_j \leq 0$ for some $1 \leq j \leq n$. Then
\[\sigma_{k}^{jj}(\kappa) \geq C(n,k)\sum_{i=1}^{n} \sigma_{k}^{ii}(\kappa)\] and 
\[\kappa_j \geq -\frac{n-k}{k}\kappa_1.\]
\end{lemma}
\begin{proof}
The first inequality follows from Lemma \ref{basic properties} by looking at
\[\sigma_{k-1}=\kappa_j\sigma_{k-1}^{jj}(\kappa)+\sigma_{k-1}(\kappa|j).\]
For the second inequality, see \cite[Lemma 11]{Ren-Wang-2}.
\end{proof}
\section{Global gradient estimates} \label{gradient estimates}
In this section, we obtain global gradient estimates when the prescribed function $\psi=\psi(x,u)$ is allowed to depend on $u$. Previously, this result was achieved by Bayard \cite[Proposition 3.1]{Bayard} for $\psi=\psi(x)$ and in a subsequent work \cite[(2.7)]{Bayard-2006}, Bayard commented that the proof could be readily extended to hold for $\psi=\psi(x,u)$. Here we not only provide a written verification and also we establish the $C^1$ estimate under the subsolution condition instead of the strict convexity assumption on $\Omega$ and $\varphi$.

We first recall the standard comparison principle.

\begin{lemma}
Let $u,v \in C^2(\Omega) \cap C(\overline{\Omega})$ be spacelike in $\Omega$. Assume $u$ is moreover $k$-admissible in $\Omega$ and the positive function $\psi(x,z) \in C^{1}(\overline{\Omega} \times \bR)$ satisfies $\psi_z \geq 0$. If $ u \leq v $ on $\partial \Omega$ and 
\[\sigma_k[u] \geq \psi(x,u), \quad \sigma_k[v] \leq \psi(x,v) \quad \text{in $\Omega$},\] then $u \leq v$ in $\Omega$.
\end{lemma}
\begin{proof}
See e.g. \cite[Theorem 5.1]{Bayard}, \cite[Lemma A]{CNS-5}, or \cite[Theorem 3.1]{Ivochkina-2}.
\end{proof}

By the work of Bartnik-Simon \cite{Bartnik-Simon}, there exists a $1$-admissible function $\overline{u}$ such that
\begin{align*}
\sigma_{1}[\overline{u}]&=n \cdot \left[\binom{n}{k}^{-1}\psi(x,\overline{u})\right]^{1/k} \quad \text{in $\Omega$,} \\
\overline{u}&=\varphi \ \quad\quad\quad\quad\quad\quad\quad\quad\quad\quad \text{on $\partial \Omega$}.
\end{align*} On the other hand, it follows from the Maclaurin's inequality that
\begin{align*}
\sigma_1[u]&=n \cdot H_1[u] \geq n \cdot \left(H_k[u]\right)^{1/k}=n \cdot \left[\binom{n}{k}^{-1}\psi(x,u)\right]^{1/k}
\end{align*} and so
\[u \leq \overline{u} \quad \text{in $\Omega$}\] by the comparison principle.

Similarly, by assuming the existence of a $k$-admissible subsolution i.e. some $k$-admissible function $\underline{u}$ such that
\begin{align*}
\sigma_{k}[\underline{u}]&\geq \psi(x,\underline{u}) \quad \text{in $\Omega$,} \\
\underline{u}&=\varphi \ \quad\quad\quad \text{on $\partial \Omega$},
\end{align*} we have
\begin{equation}
\underline{u} \leq u \leq \overline{u} \quad \text{in $\Omega$}. \label{C0 estimate}
\end{equation} 

\begin{remark}
When both $\Omega$ and $\varphi$ are strictly convex, there exists an $n$-admissible subsolution $\underline{u}$ due to the work of Delano\"{e} \cite{Delanoe}; see also \cite{Guan}.
\end{remark}

Consequently, by the Hopf lemma, for the interior normal derivative $\partial_{\gamma}$ at any point on $\partial \Omega$, we have
\[\frac{\partial \underline{u}}{\partial \gamma} \leq \frac{\partial u}{\partial \gamma} \leq \frac{\partial \overline{u}}{\partial \gamma}.\] Thus, we conclude that
\begin{equation}
\sup_{\partial \Omega} |Du| \leq \max\{\sup_{\partial \Omega} |D\underline{u}|, \max_{\partial \Omega} |D\overline{u}|\} =: 1-\theta_0 \label{boundary C1 estimate}
\end{equation} for $\theta_0 \in (0,1)$.

\begin{theorem}
Let $2 \leq k \leq n-1$. Suppose $\psi(x,z) \in C^{\infty}(\overline{\Omega} \times \bR)$ is positive and $\psi_z \geq 0$. Assume further the existence of a $k$-admissible subsolution $\underline{u}$ such that
\begin{align*}
\sigma_k[\underline{u}]&\geq \psi(x,\underline{u})\quad  \text{in $\Omega$},\\
                \underline{u}&=\varphi \ \quad\quad\quad     \text{on $\partial\Omega$}.
\end{align*}
If $u \in C^3(\Omega) \cap C^1(\overline{\Omega})$ is a $k$-admissible solution of \eqref{Dirichlet problem}, then
\begin{equation}
\sup_{\overline{\Omega}}  |Du| \leq 1-\theta \label{global gradient estimate}
\end{equation} for some $\theta \in (0,1)$ depending on $n,k, \theta_0, \diam(\Omega), \sup_{\partial \Omega}|\varphi|, \sup_{\cD} |D\psi|$ and $\inf_{\cD} \psi$ where $\cD:=\overline{\Omega} \times [\inf_{\Omega}u,\sup_{\Omega}u]$ and $\theta_0$ is the constant in \eqref{boundary C1 estimate}.
\end{theorem}
\begin{proof}
According to \eqref{boundary C1 estimate}, it is sufficient to estimate the quantity
\[\tilde{w}:=\frac{1}{w}=\frac{1}{\sqrt{1-|Du|^2}}\] and prove that
\[\sup_{\overline{\Omega}} \tilde{w} \leq C\left(1+\sup_{\partial \Omega} \tilde{w} \right)\] for some $C>0$ depending on the known constants.
As in \cite{Bayard, Guo-Jiao}, we may consider the function
\[\tilde{Q}=\tilde{w}e^{Bu},\] where $B$ is a positive constant to be determined later. If $\tilde{Q}$ achieves its maximum on $\partial \Omega$ then we have the bound 
\[\sup_{\overline{\Omega}} \tilde{w} \leq \left(\sup_{\partial \Omega} \tilde{w}\right)\cdot \exp \left[B \cdot \left(2\sup_{\partial \Omega} |\varphi|+\diam(\Omega)\right)\right]\]
and we are through.  Suppose this is not the case and $\tilde{Q}$ attains its maximum at some interior point $x_0 \in \Omega$. By rotating the standard coordinates $\{\epsilon_1,\ldots,\epsilon_{n+1}\}$ of $\bR^{n+1}$ if necessary, we may assume that
\[u_1(x_0)=|Du(x_0)|>0 \quad \text{and} \quad u_{j}(x_0)=0 \quad \text{for $j \geq 2$}.\] By further rotating $\{\epsilon_2,\ldots,\epsilon_n\}$, we may also assume $\{u_{ij}(x_0)\}$ is diagonal for $i,j \geq 2$. Pick an orthonormal frame $\{e_1,\ldots,e_n\}$ around the point $X_0=(x_0,u(x_0))$ e.g.
\[e_i=\gamma^{is}\tilde{\partial}_{s}, \quad \gamma^{is}=\delta_{is}+\frac{u_iu_s}{w(1+w)}, \quad \tilde{\partial}_{s}=\epsilon_s+u_s\epsilon_{n+1},\] so that
\[\nabla_1 u = \frac{|Du|}{w}=|\nabla u|, \quad \nabla_i u = u_i = 0 \quad \text{for $i \geq 2$}.\] Then, by the Weingarten formula, we have
\[\nabla_i \tilde{w}=-\nabla_i \langle \nu, \epsilon_{n+1}\rangle=-\langle h_{ij}e_j,\epsilon_{n+1}\rangle=-h_{ij} \langle \gamma^{js}\tilde{\partial}_{s}, \epsilon_{n+1}\rangle=h_{ij} \nabla_ju.\]

Taking the logarithm, the function
\[Q:=\log \tilde{Q}= \log \tilde{w} + Bu\] also attains its maximum at $x_0$. Consequently, at $x_0$, we have
\begin{align}
0&=\nabla_i Q=\frac{\nabla_i \tilde{w}}{\tilde{w}}+B\nabla_i u=\frac{h_{1i}\nabla_1 u}{\tilde{w}}+B\nabla_i u
\end{align} from which it follows that at $x_0$,
\begin{equation}
h_{11}=-B\tilde{w} \quad \text{and} \quad h_{1i}=0 \quad \text{for $i \geq 2$}. \label{gradient 1st critical}
\end{equation} In the remaining part of the proof, all subsequent calculations are carried out at the point $X_0$ without explicitly saying so.

Since 
\[h_{11}=\frac{u_{11}}{w^3} \quad \text{and} \quad h_{ij}=\frac{u_{ij}}{w} \quad \text{for $i,j \geq 2$},\] the matrix $\{h_{ij}\}$ is diagonal. Hence
\[F^{ij}:=\frac{\partial F}{\partial h_{ij}}=\frac{\partial \sigma_k}{\partial \kappa_i} \delta_{ij}\] and we calculate
\begin{gather} \label{gradient 2nd critical}
\begin{split}
0 & \geq F^{ii}\nabla_{ii} Q \\
&= F^{ii}\frac{\nabla_ih_{i1}\nabla_1u+h_{i1}\nabla_{i1}u}{\tilde{w}}-F^{ii} \frac{(h_{i1}\nabla_1u)^2}{\tilde{w}^2}+B F^{ii} \nabla_{ii} u\\
&=F^{ii}h_{i1i}\frac{\nabla_1u}{\tilde{w}}+F^{11}h_{11}^2-B^2 F^{11}\tilde{w}^2|Du|^2+Bk\psi \tilde{w},
\end{split}
\end{gather} where we have used \eqref{gradient 1st critical} and the Gauss formula
\[\nabla_{ij}u=-\nabla_{ij}\langle X, \epsilon_{n+1}\rangle = - \langle \nabla_{ij}X, \epsilon_{n+1}\rangle=-h_{ij}\langle \nu, \epsilon_{n+1}\rangle=\tilde{w}h_{ij}.\]
For the first term in \eqref{gradient 2nd critical}, we differentiate $F(h_{ij})=\psi$ and invoke the Codazzi equation and the Weingarten formula, 
\begin{align*}
F^{ii}h_{i1i}&=F^{ii}h_{ii1}=\nabla_1\psi=\psi_{x_{j}}\nabla_1 x_j + \psi_z \nabla_1 u + \partial_{p_k}\psi \nabla_1 \nu_{k}\\
&=\frac{\psi_{x_1}}{w}+\psi_{z}\frac{|Du|}{w} = \tilde{w}\psi_{x_1}+\tilde{w}\psi_{z}|Du| \\
&\geq \tilde{w}\psi_{x_1}.
\end{align*} 
Substituting this back into \eqref{gradient 2nd critical}, we have
\begin{align*}
0&\geq \tilde{w}\psi_{x_1}|Du|+B^2\tilde{w}^2F^{11}-B^2\tilde{w}^2|Du|^2F^{11}+Bk\psi\tilde{w} \\
&\geq -\left(|D\psi| \cdot |Du|\right)\tilde{w}+B^2F^{11}\tilde{w}^2(1-|Du|^2)+Bk\psi \tilde{w}.
\end{align*}
Since $u$ is spacelike i.e. $|Du| < 1$ and 
\[1-|Du|^2=\frac{1}{\tilde{w}^2},\] it follows that
\begin{equation}
0 \geq \tilde{w} \left(Bk\psi - |D\psi|\right). \label{gradient 3rd critical}
\end{equation}
Therefore, if we had chosen e.g.
\[B> \frac{\sup_{\cD} |D\psi|}{k \cdot \inf_{\cD} \psi},\] then the right-hand side of \eqref{gradient 3rd critical} would be strictly positive and a contradiction would arise. Thus, we conclude that the maximum must be attained on the boundary and
\[\sup_{\overline{\Omega}} \tilde{w} \leq \left(\sup_{\partial \Omega} \tilde{w}\right) \exp\left[\frac{\sup_{\cD} |D\psi|}{k\cdot \inf_{\cD} \psi} \left(2 \sup_{\partial \Omega} \varphi + \diam(\Omega)\right)\right].\] The proof is now complete.
\end{proof}

\section{The scalar curvature equation} \label{k=2 proof}
The curvature bound for the scalar curvature equation was first obtained by Bayard \cite{Bayard} in dimension three for $\psi=\psi(x)$ and in dimension four for $\psi=\const$. It was then extended by Urbas \cite{Urbas} to all dimensions for $\psi=\psi(x,u)$. We remark that Urbas' proof in fact works for the case $\psi=\psi(x,u,Du)$ as well. However, as Urbas commented in \cite{Urbas}, the proof relied crucially on the additional assumption that the boundary data $\varphi$ is uniformly convex i.e. $D^2\varphi \geq c_0 I$ uniformly for some positive constant $c_0$. The cruciality of this condition is illustrated in our Remark \ref{cruciality of uniform convexity} below.

In this section, we improve the curvature estimates for the scalar curvature equation by removing the assumption of uniform convexity. 

\begin{theorem}
Let $\varphi \in C^2(\Omega) \cap C(\overline{\Omega})$ be spacelike and affine. Suppose $\psi(x,z,p) \in C^{\infty}(\overline{\Omega} \times \bR \times \bR^n)$ is positive. If $u \in C^4(\Omega) \cap C^2(\overline{\Omega})$ is a $2$-admissible solution to
\begin{align*}
       \sigma_2[u]&=\psi(x,u, Du)\quad  \text{in $\Omega$},\\
                u&=\varphi \quad\quad\quad\quad\quad     \text{on $\partial\Omega$},
\end{align*}
then the maximum principal curvature 
\[\kappa_{\max}(x):=\max_{1 \leq i \leq n} \kappa_i(x)\]
of its graph $\Sigma$ satisfies
\[\sup_{\Omega'}\kappa_{\max} \leq C(\Omega')\] for any $\Omega' \subset \subset \Omega$, where $C(\Omega')>0$ depends on $n, \theta, \norm{\psi}_{C^2(\cD)}, \inf_{\cD} \psi,$ and $\norm{\varphi}_{C^1(\overline{\Omega})}$; here $\theta$ is the constant in \eqref{global gradient estimate} and 
\[\cD:=\overline{\Omega} \times [\inf_{\Omega} u, \sup_{\Omega} u] \times \bR^n.\]
\end{theorem}
\begin{remark}
It is sufficient to assume $\varphi$ is spacelike and satisfies $\sigma_2[\varphi]<\sigma_2[u]$ in $\Omega$.
\end{remark}
\begin{proof}
For convenience in this proof only, we write the scalar curvature equation in the following form
\[F[u]:=\sigma_{2}^{1/2}[u]=\psi(x,u,Du),\] to better invoke the concavity property of the $\sigma_{2}^{1/2}$ operator.

Consider the quantity
\[\widetilde{W}(X,\xi)=\eta^{\beta}h_{\xi\xi}\exp\left(\frac{\alpha}{2}|X|^2\right) \] for $X \in \Sigma$ and $\xi \in T_{X}\Sigma$, where $\eta=\varphi-u$, and $\alpha, \beta >0$ are some large constants to be chosen later.

Suppose the maximum of $\widetilde{W}$ is attained at some interior point $X_0=(x_0,u(x_0)) \in \Sigma$ and some $\xi_0 \in T_{X_0}\Sigma$. We choose a local orthonormal frame $\{e_1,\ldots,e_n\}$ around $X_0$ such that
\[\xi_0=e_1, \quad \nabla_{e_i}e_j=0 \quad \text{at $X_0$}.\] We may also assume the second fundamental form $h_{ij}=\kappa_i\delta_{ij}$ is diagonal at $X_0$ with principal curvatures ordered as
\[\kappa_1 \geq \kappa_2 \geq \cdots \geq \kappa_n.\] Let $\zeta=e_1$. Then the function
\[W(X)=\eta^{\beta}h_{ab}\zeta_a\zeta_b\] is defined near $X_0$ and attains an interior maximum also at $X_0$. By our special choice of the frame, we find that $Z:=h_{ab}\zeta_a\zeta_b$ satisfies
\[\nabla_iZ=\nabla_ih_{11} \quad \text{and} \quad \nabla_i\nabla_j Z = \nabla_i \nabla_j h_{11} \quad \text{at $X_0$}.\] Therefore, by working with $\log W$, we obtain at $X_0$
\begin{align}
0&=\beta \frac{\nabla_i \eta}{\eta} + \frac{h_{11i}}{h_{11}}+\alpha\langle X,e_i\rangle, \label{k=2 1st critical 1}\\
0&\geq \beta\frac{\nabla_{ii}\eta}{\eta}-\beta \left(\frac{\nabla_i\eta}{\eta}\right)^2 +\frac{h_{11ii}}{\kappa_1}-\frac{h_{11i}^2}{\kappa_{1}^2}+\alpha \left(1+h_{ii}\langle X, \nu \rangle \right). \label{k=2 2nd critical 1}
\end{align} In what follows, we will carry out all calculations at the point $X_0$ without explicitly indicating so. Now, we shall contract \eqref{k=2 2nd critical 1} with
\[F^{ij}:=\frac{\partial (\sigma_{2}^{1/2})}{\partial h_{ij}}\] and estimate it term by term.

Recall the interchanging formula \cite[Lemma 2.1]{Urbas}, we have
\begin{gather} \label{interchanging formula}
\begin{split}
F^{ij}\nabla_i\nabla_j h_{11}=&-F^{ij,kl}\nabla_1 h_{ij}\nabla_1 h_{kl}-F^{ij}h_{ij}\kappa_{1}^2\\
&+F^{ij}h_{ik}h_{jk}\kappa_1+\nabla_1\nabla_1 \psi \\
=&-F^{ij,kl}\nabla_1 h_{ij}\nabla_1 h_{kl}-F\kappa_{1}^2 \\
 &+\kappa_1 \sum_{i=1}^{n} F^{ii}\kappa_{i}^2 + \nabla_1\nabla_1\psi.
\end{split}
\end{gather} For the first term, we apply \cite[Lemma 2.2]{Urbas} in addition to concavity of $\sigma_{2}^{1/2}$ to obtain that
\[-F^{ij,kl}h_{ij1}h_{kl1}=\sum_{i \neq j} \frac{F^{ii}-F^{jj}}{\kappa_i-\kappa_j}h_{ij1}^2\geq 2\sum_{i \neq 1} \frac{F^{ii}-F^{11}}{\kappa_1-\kappa_i}h_{11i}^2.\] For the last term in \eqref{interchanging formula}, since $\psi$ contains the gradient, we have
\begin{equation}
\nabla_1\nabla_1\psi \geq -C(1+\kappa_{1}^2). \label{differentiate twice}
\end{equation} Next, we consider the term $F^{ii}\nabla_{ii}\eta$. To compute this we may extend $\varphi$ to be constant in the $\epsilon_{n+1}$ direction. Then, since $u=X_{n+1}$ on $\Sigma$, we have
\begin{align*}
\nabla_i\nabla_j \eta&=\sum_{\alpha=1}^{n} \frac{\partial \varphi}{\partial X_{\alpha}} \nabla_i\nabla_j X_\alpha+\sum_{\alpha,\gamma=1}^{n} \frac{\partial^2\varphi}{\partial X_\alpha\partial X_\gamma} \nabla_iX_\alpha\nabla_j X_\gamma - \nabla_i\nabla_jX_{n+1} \\
&=\sum_{\alpha=1}^{n} \frac{\partial \varphi}{\partial X_\alpha}h_{ij}\nu_{\alpha}+\sum_{\alpha,\gamma=1}^{n}\frac{\partial^2\varphi}{\partial X_{\alpha} \partial X_{\gamma}} \nabla_i X_\alpha \nabla_j X_\gamma - h_{ij}\nu_{n+1}
\end{align*} from which it follows that
\begin{equation}
F^{ii}\nabla_{ii}\eta \geq \left(\sum_{\alpha=1}^{n} \frac{\partial \varphi}{\partial X_\alpha}\nu_\alpha-\nu_{n+1}\right)F^{ii}h_{ii} \geq -C. \label{differentiate the cutoff twice}
\end{equation}

Substituting all these back into \eqref{k=2 2nd critical 1}, we get
\begin{gather} \label{k=2 2nd critical 2}
\begin{split}
0 \geq & 2\sum_{i \neq 1} \frac{F^{ii}-F^{11}}{\kappa_1-\kappa_i} \frac{h_{11i}^2}{\kappa_1}- \sum_{i=1}^{n} F^{ii}\frac{h_{11i}^2}{\kappa_{1}^2}+\sum_{i=1}^{n} F^{ii}\kappa_{i}^2+\alpha \sum_{i=1}^{n} F^{ii}\\
& -\beta \sum_{i=1}^{n} F^{ii}\frac{|\nabla_i \eta|^2}{\eta^2} - \frac{C\beta}{\eta} -C\kappa_1-C.
\end{split}
\end{gather} Let $\delta>0$ be a number whose value is to be determined. We proceed by considering two cases.

\bigskip

\textbf{Case 1: $\kappa_n \leq -\delta \kappa_1$.}

In this case, we may use the fact that $|\nabla \eta| \leq C$ and the first order critical condition \eqref{k=2 1st critical 1} to estimate
\begin{align*}
&\ \sum_{i=1}^{n} F^{ii}\frac{h_{11i}^2}{\kappa_{1}^2}+\beta \sum_{i=1}^{n} F^{ii}\frac{|\nabla_i \eta|^2}{\eta^2}\\
\leq &\ C\alpha^2\sum_{i=1}^{n} F^{ii}+\frac{C\beta(1+\beta)}{\eta^2}\sum_{i=1}^{n} F^{ii}.
\end{align*} On the other hand, since $\kappa_n \leq -\delta \kappa_1$, we have
\[\sum_{i=1}^{n} F^{ii}\kappa_{i}^2 \geq F^{nn}\kappa_{n}^2 \geq C(n,k) \delta^2 \kappa_{1}^2 \sum_{i=1}^{n} F^{ii}\] by lemma \ref{negative kappa}. Hence, the inequality \eqref{k=2 2nd critical 2} becomes
\begin{gather} \label{k=2 2nd critical 3}
\begin{split}
0 \geq & \left(C\delta^2 \kappa_{1}^2 -\frac{C\beta(1+\beta)}{\eta^2}-C\alpha^2\right)\sum_{i=1}^{n} F^{ii} - \frac{C\beta}{\eta}-C\kappa_1-C
\end{split}
\end{gather} and a bound for $\eta^{b}\kappa_1$ at $X_0$ follows from this.

\bigskip

\textbf{Case 2: $\kappa_n \geq -\delta \kappa_1$.}

This time, in order to handle the negative third order terms, we partition the indices $\{1,\ldots,n\}$ into
\[I=\{j: F^{jj} \leq \theta^{-1} F^{11}\}, \quad J=\{j: F^{jj}>\theta^{-1} F^{11}\}\] for some $\theta>0$ to be determined. Again, we use the first critical condition \eqref{k=2 1st critical 1} and the fact that $|\nabla \eta| \leq C$ to estimate
\begin{gather} \label{estimate of the third order terms}
\begin{split}
&\ \sum_{i=1}^{n} F^{ii}\frac{h_{11i}^2}{\kappa_{1}^2}+\beta \sum_{i=1}^{n} F^{ii}\frac{|\nabla_i \eta|^2}{\eta^2}\\
= &\ \left(\sum_{I}+\sum_{J}\right) F^{ii}\frac{h_{11i}^2}{\kappa_{1}^2}+\beta \left(\sum_{I}+\sum_{J}\right) F^{ii}\frac{|\nabla_i \eta|^2}{\eta^2} \\
\leq &\ \left(\sum_{I} C\beta^3 F^{ii}\frac{|\nabla_i\eta|^2}{\eta^2} + \frac{C\alpha^2}{\beta} F^{ii}\right)+\sum_{J} F^{ii}\frac{h_{11i}^2}{\kappa_{1}^2} \\
     &\ +\beta \sum_{I} F^{ii} \frac{|\nabla_i\eta|^2}{\eta^2} + \left(\sum_{J} \frac{C}{\beta} F^{ii}\frac{h_{11i}^2}{\kappa_{1}^2}+\frac{C\alpha^2}{\beta}F^{ii} \right) \\
\leq &\ C(\beta+\beta^3)\sum_{I} F^{ii} \frac{|\nabla_i \eta|^2}{\eta^2}+\left(1+\frac{C}{\beta}\right)\sum_{J} F^{ii}\frac{h_{11i}^2}{\kappa_{1}^2}+\frac{C\alpha^2}{\beta}\sum_{i=1}^{n} F^{ii}\\
\leq &\ C(\beta+\beta^3)\frac{F^{11}}{\eta^2}+\left(1+\frac{C}{\beta}\right)\sum_{J} F^{ii}\frac{h_{11i}^2}{\kappa_{1}^2}+\frac{C\alpha^2}{\beta}\sum_{i=1}^{n} F^{ii},
\end{split}
\end{gather} where we have used the Cauchy's inequality
\[\frac{h_{11i}^2}{\kappa_{1}^2} \leq C\epsilon \alpha^2 + \frac{\beta^2}{C\epsilon}\frac{|\nabla_{i}\eta|^2}{\eta^2}\] with $\epsilon=1/\beta$.

With \eqref{estimate of the third order terms} at hand,  the inequality \eqref{k=2 2nd critical 2} becomes
\begin{gather} \label{k=2 2nd critical 4}
\begin{split}
0 \geq &\ 2\sum_{i \neq 1} \frac{F^{ii}-F^{11}}{\kappa_1-\kappa_i} \frac{h_{11i}^2}{\kappa_1}- \left(1+\frac{C}{\beta}\right)\sum_{J} F^{ii}\frac{h_{11i}^2}{\kappa_{1}^2}+\left(\alpha-\frac{C\alpha^2}{\beta}\right) \sum_{i=1}^{n} F^{ii}\\
& +F^{11}\kappa_{1}^2-\frac{C(\beta+\beta^3)}{\eta^2} F^{11} - \frac{C\beta}{\eta} -C\kappa_1-C.
\end{split}
\end{gather} Note that
\begin{align*}
&\ 2\sum_{i \neq 1} \frac{F^{ii}-F^{11}}{\kappa_1-\kappa_i} \frac{h_{11i}^2}{\kappa_1}- \left(1+\frac{C}{\beta}\right)\sum_{J} F^{ii}\frac{h_{11i}^2}{\kappa_{1}^2} \\
\geq &\ 2\sum_{J} \frac{F^{ii}-F^{11}}{\kappa_1-\kappa_i} \frac{h_{11i}^2}{\kappa_1}- \left(1+\frac{C}{\beta}\right)\sum_{J} F^{ii}\frac{h_{11i}^2}{\kappa_{1}^2} \\
\geq &\ \sum_{J} \left[2\cdot \frac{1-\theta}{1+\delta}-\left(1+\frac{C}{\beta}\right)\right] F^{ii}\frac{h_{11i}^2}{\kappa_{1}^2}
\end{align*} can be made non-negative if we choose appropriate values for the parameters $\delta, \theta, \beta$. Indeed, let 
\[\frac{C}{\beta}=\varepsilon, \quad \frac{1-\theta}{1+\delta} \geq 1-\varepsilon,\] and for $\beta \geq 1$ large enough so that $0 < \varepsilon \leq 1/3$, we have
\[2\cdot \frac{1-\theta}{1+\delta}-\left(1+\frac{C}{\beta}\right) \geq 1-3\varepsilon \geq 0.\] Consequently, we are left with
\[0 \geq \left(\alpha-\frac{C\alpha^2}{\beta}\right) \sum_{i=1}^{n} F^{ii}+F^{11}\kappa_{1}^2-\frac{C(\beta+\beta^3)}{\eta^2} F^{11} - \frac{C\beta}{\eta} -C\kappa_1-C.\]
Finally, there is only one troublesome term i.e. $-C\kappa_1$ to be handled, which arose from the interchanging formula \eqref{interchanging formula} and the twice differentiation \eqref{differentiate twice} of $\psi$. Indeed, by fixing $\beta \geq C\alpha^2$ is large enough, we have
\[\left(\alpha-\frac{C\alpha^2}{\beta}\right) \sum_{i=1}^{n} F^{ii} \geq C\alpha \sum_{i=1}^{n} F^{ii} \geq C\alpha \kappa_1\] and a bound for $\eta^\beta \kappa_1$ at $X_0$ is achieved by choosing a large enough $\alpha$.

The proof is now complete.
\end{proof}
Several remarks are in order.

\begin{remark} \label{cruciality of uniform convexity}
In \cite{Urbas}, Urbas used the additional assumption that the boundary data $\varphi$ is uniformly convex to obtain that
\begin{equation}
\sum_{i,j} F^{ij}\nabla_i\nabla_j \eta  \geq c_0\sum_{i=1}^{n} F^{ii}-C
\end{equation} instead of \eqref{differentiate the cutoff twice}. Then, by further invoking the property that
\begin{equation}
\sum_{i=1}^{n} F^{ii} \geq C\kappa_1, \label{the property}
\end{equation} Urbas was able to handle the term $-C\kappa_1$ with the term
\[\beta\frac{F^{ij}\nabla_i\nabla_j \eta}{\eta}\geq \frac{C\beta}{\eta} \kappa_1\] by fixing $\beta \geq 1$ large.
\end{remark}
\begin{remark}
Both our proof and the proof of Urbas cannot be generalized to $k \geq 3$ because the property \eqref{the property} holds only when $k=2$; see the demonstration in \cite[Remark at the end of page 315]{Urbas}. 
\end{remark}

By an almost identical proof, we obtain the following and also Theorem \ref{k=2}.
\begin{theorem}
Let $n \geq 3$ and let $\Omega \subseteq \bR^n$ be a convex $1$-admissible domain with a smooth boundary. Assume $\psi(x,u,Du) \in C^{2}(\overline{\Omega} \times \bR \times \bR^n)$ is positive and $\varphi \in C^2(\overline{\Omega})$ is spacelike. Suppose $u \in C^{4}(\Omega) \cap C^2(\overline{\Omega})$ is a $2$-admissible solution to 
\begin{align*}
       \sigma_2[u]&=\psi(x,u, Du)\quad  \text{in $\Omega$},\\
                u&=\varphi \quad\quad\quad\quad\quad     \text{on $\partial\Omega$}.
\end{align*}
Then the maximum principal curvature 
\[\kappa_{\max}(x):=\max_{1 \leq i \leq n} \kappa_{i}(x)\] of its graph satisfies
\[\max_{\Omega} \kappa_{\max}(x) \leq C\left(1+\max_{\partial \Omega} \kappa_{\max}(x)\right)\] for some $C>0$ depending on $n,\norm{u}_{C^1(\overline{\Omega})}, \norm{\psi}_{C^2(\cD)}$ and $\norm{\varphi}_{C^1(\overline{\Omega})}$, where
\[\cD:=\overline{\Omega} \times [\inf_{\Omega} u, \sup_{\Omega} u] \times \bR^n.\]
\end{theorem}

\section{The $k$-curvature equation} \label{curvature estimate}
Although the curvature estimates are now known for $k=1,2,n$ due to the work of Bartnik-Simon \cite{Bartnik-Simon}, Bayard \cite{Bayard}, Delanoe \cite{Delanoe}, Urbas \cite{Urbas}, and most recently for $k=n-1,n-2$ due to the work of Ren-Wang \cite{Ren-Wang-1, Ren-Wang-2}, the question of whether the bound remains valid for $3 \leq k \leq n-3$ is still open to this date. For this direction of research, some partial progress has been obtained: In \cite{Huang}, Huang proved the curvature bound for the $k$-curvature equation $\sigma_k[u]=\psi(X,\tilde{w})$ when the right-hand side is convex in $\tilde{w}$ and 
\begin{equation}
\frac{\partial \psi^{1/k}(X,\tilde{w})}{\partial \tilde{w}} \cdot \tilde{w} \geq \psi^{1/k}(X,\tilde{w}) \quad \text{for fixed $X=(x,u(x))$}, \label{Huang's assumption}
\end{equation} where
\[\tilde{w}:=\frac{1}{w}=\frac{1}{\sqrt{1-|Du|^2}}.\] The approach of Huang may be appealing for some particular functions $\psi$, but the assumption excludes the very desirable case when e.g. $\psi=\psi(X)$ does not contain gradient terms at all. In \cite{Wang}, Z.~Wang obtained the bound for a general right-hand side $\psi=\psi(x,u,Du)$ and all $k$ if the solution graph is additionally $(k+1)$-convex.

In this section, we establish the curvature bound for admissible solutions to the $k$-curvature equation which additionally have semi-convex graphs. Since semi-convexity is weaker than $(k+1)$-convexity \cite[Lemma 2.13]{Bin-1}, our result generalizes the estimate due to Z.~Wang. Our proof adapts the arguments from a work of Lu \cite{Lu}, where the core ideas are inspired by the novel paper of Guan-Ren-Wang \cite{Guan-Ren-Wang}. 

\begin{theorem}
Let $\varphi \in C^2(\Omega) \cap C(\overline{\Omega})$ be spacelike and affine. Suppose $\psi(x,z,p) \in C^{\infty}(\overline{\Omega} \times \bR \times \bR^n)$ is positive. If $u \in C^4(\Omega) \cap C^2(\overline{\Omega})$ is a $k$-admissible solution to 
\begin{align*}
       \sigma_k[u]&=\psi(x,u, Du)\quad  \text{in $\Omega$},\\
                u&=\varphi \quad\quad\quad\quad\quad     \text{on $\partial\Omega$},
\end{align*}
and its graph $\Sigma$ is semi-convex i.e. the principal curvatures $\kappa=(\kappa_1,\ldots,\kappa_n)$ of $\Sigma$ satisfy
\[\kappa_{i}(x) \geq -K \quad \text{for all $x \in \overline{\Omega}$ and all $1 \leq i \leq n$},\] then the maximum principal curvature 
\[\kappa_{\max}(x):=\max_{1 \leq i \leq n} \kappa_i(x)\] of $\Sigma$ satisfies
\[\sup_{\Omega'}\kappa_{\max} \leq C(\Omega')\] for any $\Omega' \subset \subset \Omega$, where $C(\Omega')>0$ depends on $n, k, K, \theta, \norm{\psi}_{C^2(\cD)}, \inf_{\cD} \psi,$ and $\norm{\varphi}_{C^1(\overline{\Omega})}$; here $\theta$ is the constant in \eqref{global gradient estimate} and 
\[\cD:=\overline{\Omega} \times [\inf_{\Omega} u, \sup_{\Omega} u] \times \bR^n.\]
\end{theorem}

\begin{proof}
Denote by 
\[\tilde{w}:=\frac{1}{w}=\frac{1}{\sqrt{1-|Du|^2}}\] and let 
\[\eta=\varphi-u.\]
The setting is exactly the same as in section \ref{k=2 proof} and we shall dive directly into the calculations. This time we use the following test function
\[W=\beta \log \eta + \log \kappa_{\max} + N\tilde{w} + \frac{\alpha}{2}|X|^2,\] where $\beta, N, \alpha>0$ are possibly large constants to be chosen later. The key difference here is that we are going to utilize almost all available positive terms which were plausibly omitted in section \ref{k=2 proof}.

Suppose $W$ attains its maximum at some interior point $X_0=(x_0,u(x_0))$. We may choose a local orthonormal frame $\{e_1,\ldots,e_n\}$ around $X_0$ such that the second fundamental form $h_{ij}=\kappa_i\delta_{ij}$ is diagonalized and 
\[\kappa_{\max}=\kappa_1 \geq \kappa_2 \geq \cdots \geq \kappa_n.\] In case $\kappa_1$ has multiplicity $m>1$ i.e.
\[\kappa_1=\kappa_2=\cdots=\kappa_m>\kappa_{m+1} \geq \cdots \geq \kappa_n,\] we may apply a smooth approximation lemma due to Brendle-Choi-Daskalopoulos \cite{BCD} to obtain that
\begin{align}
\delta_{kl}\cdot (\kappa_1)_i&=h_{kli}, \quad 1 \leq k,l \leq m, \label{approximation}\\
(\kappa_1)_{ii}&\geq h_{11ii}+2\sum_{p>m}\frac{h_{1pi}^2}{\kappa_1-\kappa_p}, \nonumber
\end{align} in the viscosity sense. Then at $x_0$, we have
\begin{align}
0&=\beta\frac{\nabla_i\eta}{\eta}+\frac{(\kappa_1)_i}{\kappa_1}+N\nabla_i \tilde{w}+\alpha\langle X,e_i\rangle=\beta \frac{\nabla_i\eta}{\eta}+\frac{h_{11i}}{\kappa_1}+N\nabla_i \tilde{w}+\alpha\langle X,\hat{e}_i\rangle, \label{1st critical}\\
0&\geq \beta \frac{\nabla_{ii}\eta}{\eta}-\beta \left(\frac{\nabla_i\eta}{\eta}\right)^2+\frac{(\kappa_1)_{ii}}{\kappa_1}-\frac{(\kappa_1)_{i}^{2}}{\kappa_{1}^2}+N\nabla_{ii}\tilde{w}+\alpha(1+h_{ii}\langle X,\nu \rangle) \nonumber\\
&\geq \beta \frac{\nabla_{ii}\eta}{\eta}-\beta \left(\frac{\nabla_i\eta}{\eta}\right)^2+\frac{h_{11ii}}{\kappa_1}+2\sum_{p>m}\frac{h_{1pi}^2}{\kappa_1(\kappa_1-\kappa_p)}-\frac{h_{11i}^2}{\kappa_{1}^2}+N\nabla_{ii}\tilde{w}+\alpha(1+h_{ii}\langle X,\nu \rangle). \label{2nd critical 1}
\end{align}

Contracting \eqref{2nd critical 1} with $F=\sigma_k$, we have
\begin{gather}\label{2nd critical 2}
\begin{split}
0 \geq &\ \sum_{i=1}^{n} \frac{F^{ii}h_{11ii}}{\kappa_1}+2\sum_{i=1}^{n}\sum_{p>m}\frac{F^{ii}h_{1pi}^2}{\kappa_1(\kappa_1-\kappa_i)}-\sum_{i=1}^{n} \frac{F^{ii}h_{11i}^2}{\kappa_{1}^2}\\
& + \frac{\beta}{\eta}\sum_{i=1}^{n} F^{ii}\nabla_{ii}\eta - \beta \sum_{i=1}^{n} F^{ii} \left(\frac{\nabla_i\eta}{\eta}\right)^2 +N \sum_{i=1}^{n} F^{ii}w_{ii}+\alpha \sum_{i=1}^{n} F^{ii}. 
\end{split}
\end{gather}
Now, in the Minkowski space, the interchanging formula reads
\begin{equation}
h_{11ii}=h_{ii11}+h_{11}h_{ii}^2-h_{11}^2h_{ii}, \label{interchanging formula 1}
\end{equation} and so
\[F^{ii}h_{11ii}=F^{ii}h_{ii11}+\kappa_1F^{ii}\kappa_{i}^2-\kappa_{1}^2F^{ii}\kappa_i.\] Differentiating the equation $F=\psi$ twice yields
\begin{equation}
\sum_{i=1}^{n} F^{ii}h_{ii1}=\nabla_1\psi \label{differentiate once}
\end{equation} and
\begin{equation}
\sum_{p,q,r,s}F^{pq,rs}h_{pq1}h_{rs1}+\sum_{i=1}^{n} F^{ii}h_{ii11}=\nabla_{1}\nabla_1\psi \geq -C(1+\kappa_{1}^2)
\end{equation}
The inequality \eqref{2nd critical 2} becomes
\begin{gather}\label{2nd critical 3}
\begin{split}
0 \geq &\ -\sum_{p,q,r,s} \frac{F^{pq,rs}h_{pq1}h_{rs1}}{\kappa_1}+\sum_{i=1}^{n} F^{ii}\kappa_{i}^2-C\kappa_1-C\\
&\ +2\sum_{i=1}^{n}\sum_{p>m}\frac{F^{ii}h_{1pi}^2}{\kappa_1(\kappa_1-\kappa_i)}-\sum_{i=1}^{n} \frac{F^{ii}h_{11i}^2}{\kappa_{1}^2}\\
&\ +\frac{\beta}{\eta}\sum_{i=1}^{n} F^{ii}\nabla_{ii}\eta - \beta \sum_{i=1}^{n} F^{ii} \left(\frac{\nabla_i\eta}{\eta}\right)^2+N\sum_{i=1}^{n} F^{ii}w_{ii}+\alpha \sum_{i=1}^{n} F^{ii}. 
\end{split}
\end{gather} We continue to expand the terms in \eqref{2nd critical 3}. First, we have
\begin{align*}
-\sum_{p,q,r,s}\frac{F^{pq,rs}h_{pq1}h_{rsq}}{\kappa_1}&=-\sum_{p,q}\frac{F^{pp,qq}h_{pp1}h_{qq1}}{\kappa_1}+\sum_{p,q} \frac{F^{pp,qq}h_{pq1}^2}{\kappa_1}\\
&\geq -\sum_{p,q}\frac{F^{pp,qq}h_{pp1}h_{qq1}}{\kappa_1}+2\sum_{i>m}\frac{F^{11,ii}h_{11i}^2}{\kappa_1}\\
&=-\sum_{p,q}\frac{F^{pp,qq}h_{pp1}h_{qq1}}{\kappa_1}+2\sum_{i>m}\frac{F^{ii}-F^{11}}{\kappa_1(\kappa_1-\kappa_i)}h_{11i}^2.
\end{align*} On the other hand, we have
\begin{align*}
2\sum_{i=1}^{n}\sum_{p>m}\frac{F^{ii}h_{1pi}^2}{\kappa_1(\kappa_1-\kappa_p)}&\geq 2\sum_{p>m}\frac{F^{pp}h_{1pp}^2}{\kappa_1(\kappa_1-\kappa_p)}+2\sum_{p>m}\frac{F^{11}h_{1p1}^2}{\kappa_1(\kappa_1-\kappa_p)}\\
&=2\sum_{i>m} \frac{F^{ii}h_{ii1}^2}{\kappa_1(\kappa_1-\kappa_i)}+2\sum_{i>m}\frac{F^{11}h_{11i}^2}{\kappa_1(\kappa_1-\kappa_i)}.
\end{align*} Substituting these back into \eqref{2nd critical 3}, we have
\begin{gather} \label{2nd critical 4}
\begin{split}
0\geq &-\sum_{p,q} \frac{F^{pp,qq}h_{pp1}h_{qq1}}{\kappa_1}-\frac{F^{11}h_{111}^2}{\kappa_{1}^2}+2\sum_{i>m}\frac{F^{ii}h_{ii1}^2}{\kappa_1(\kappa_1-\kappa_i)}-C\kappa_1-C\\
&+2\sum_{i>m}\frac{F^{ii}-F^{11}}{\kappa_1(\kappa_1-\kappa_i)}h_{11i}^2+2\sum_{i>m}\frac{F^{11}h_{11i}^2}{\kappa_1(\kappa_1-\kappa_i)}-\sum_{i \neq 1} \frac{F^{ii}h_{11i}^2}{\kappa_{1}^2}\\
&+\frac{\beta}{\eta}\sum_{i=1}^{n} F^{ii}\nabla_{ii}\eta - \beta \sum_{i=1}^{n} F^{ii} \left(\frac{\nabla_i\eta}{\eta}\right)^2+\sum_{i=1}^{n} F^{ii}\kappa_{i}^2 +N\sum_{i=1}^{n} F^{ii}\tilde{w}_{ii}+ \alpha \sum_{i=1}^{n} F^{ii}.
\end{split}
\end{gather} We now look at the terms involving $\eta$. By the same calculation as in \eqref{differentiate the cutoff twice}, we have
\[\sum_{i=1}^{n} F^{ii} \nabla_{ii}\eta \geq -C.\] Also, invoking $|\nabla \eta| \leq C$ and the first critical condition \eqref{1st critical}, we have
\begin{align*}
\beta \sum_{i=1}^{n} F^{ii} \left(\frac{\nabla_i\eta}{\eta}\right)^2&=\beta F^{11} \left(\frac{\nabla_1\eta}{\eta}\right)^2 + \beta \sum_{i \neq 1} F^{ii} \left(\frac{\nabla_i\eta}{\eta}\right)^2 \\
&\leq \frac{C\beta}{\eta^2}F^{11} + \sum_{i \neq 1} \left(\frac{C}{\beta} F^{ii}\frac{h_{11i}^2}{\kappa_{1}^2} + \frac{CN^2}{\beta} F^{ii} (\nabla_i\tilde{w})^2+\frac{C\alpha^2}{\beta} F^{ii}\right).
\end{align*} That is, \eqref{2nd critical 4} reduces to
\begin{gather} \label{2nd critical 5}
\begin{split}
0\geq &-\sum_{p,q} \frac{F^{pp,qq}h_{pp1}h_{qq1}}{\kappa_1}-\frac{F^{11}h_{111}^2}{\kappa_{1}^2}+2\sum_{i>m}\frac{F^{ii}h_{ii1}^2}{\kappa_1(\kappa_1-\kappa_i)}\\
&+2\sum_{i>m}\frac{F^{ii}-F^{11}}{\kappa_1(\kappa_1-\kappa_i)}h_{11i}^2+2\sum_{i>m}\frac{F^{11}h_{11i}^2}{\kappa_1(\kappa_1-\kappa_i)}-\left(1+\frac{C}{\beta}\right)\sum_{i \neq 1} \frac{F^{ii}h_{11i}^2}{\kappa_{1}^2}\\
&+\left(N\sum_{i=1}^{n} F^{ii}\tilde{w}_{ii}-\frac{CN^2}{\beta}\sum_{i=1}^{n} F^{ii} \tilde{w}_{i}^2 \right)+ \left(\alpha -\frac{C\alpha^2}{\beta}\right)\sum_{i=1}^{n} F^{ii}\\
&+\left(F^{11}\kappa_{1}^2-\frac{C\beta}{\eta^2}F^{11}\right)-\frac{C\beta}{\eta}-C\kappa_1-C.
\end{split}
\end{gather} By choosing $\beta>0$ large enough, we may assume the number
\[\varepsilon:=\frac{C}{\beta}\] is sufficiently small. Also, recall the following formulas \cite[(3.1) and (3.13)]{Urbas-CAG},
\begin{align*}
F^{ij}\nabla_i\tilde{w}\nabla_j\tilde{w}& \leq |Du|^2\tilde{w}^2F^{ij}h_{ik}h_{jk}, \\
F^{ij}\nabla_i\nabla_j\tilde{w}&=\tilde{w}F^{ij}h_{im}h_{jm}+\langle \nabla \psi, \epsilon_{n+1}\rangle,
\end{align*} where $\epsilon_{n+1}$ is the $(n+1)$-th standard coordinate in $\bR^{n+1}$. We are left with
\begin{gather} \label{2nd critical 6}
\begin{split}
0\geq &-\sum_{p,q} \frac{F^{pp,qq}h_{pp1}h_{qq1}}{\kappa_1}-\frac{F^{11}h_{111}^2}{\kappa_{1}^2}+2\sum_{i>m}\frac{F^{ii}h_{ii1}^2}{\kappa_1(\kappa_1-\kappa_i)}\\
&+2\sum_{i>m}\frac{F^{ii}-F^{11}}{\kappa_1(\kappa_1-\kappa_i)}h_{11i}^2+2\sum_{i>m}\frac{F^{11}h_{11i}^2}{\kappa_1(\kappa_1-\kappa_i)}-\left(1+\varepsilon\right)\sum_{i \neq 1} \frac{F^{ii}h_{11i}^2}{\kappa_{1}^2}\\
&+CN\sum_{i=1}^{n} F^{ii}\kappa_{i}^2+ C\alpha\sum_{i=1}^{n} F^{ii} -\frac{C\beta}{\eta}-C\kappa_1-CN-C,\\
\end{split}
\end{gather} where we have also assumed $\eta^2\kappa_{1}^2 \geq C\beta$ is sufficiently large. It remains only to handle the first two lines. Indeed, the second line could be easily handled as
\[h_{11i}=h_{1i1}=\delta_{1i} \cdot (\kappa_{1})_{1} = 0, \quad 1<i \leq m \quad \text{by \eqref{approximation}}\] and so
\begin{gather} \label{the second line}
\begin{split}
&\ 2\sum_{i>m}\frac{F^{ii}-F^{11}}{\kappa_1(\kappa_1-\kappa_i)}h_{11i}^2+2\sum_{i>m}\frac{F^{11}h_{11i}^2}{\kappa_1(\kappa_1-\kappa_i)}-\left(1+\varepsilon\right)\sum_{i \neq 1} \frac{F^{ii}h_{11i}^2}{\kappa_{1}^2} \\
=&\ 2\sum_{i>m} \frac{F^{ii}h_{11i}^2}{\kappa_1(\kappa_1-\kappa_i)}-(1+\varepsilon)\sum_{i>m} \frac{F^{ii}h_{11i}^2}{\kappa_{1}^2} \\
=&\ \sum_{i>m} \frac{F^{ii}h_{11i}^2}{\kappa_{1}^2(\kappa_1-\kappa_i)} \left[2\kappa_1-(1+\varepsilon)(\kappa_1-\kappa_i)\right]\\
=&\ \sum_{i>m} \frac{F^{ii}h_{11i}^2}{\kappa_{1}^2(\kappa_1-\kappa_i)}\left[(1-\varepsilon)\kappa_1+(1+\varepsilon)\kappa_i\right]
\end{split}
\end{gather} is non-negative by semi-convexity $\kappa_{i} \geq -K$. Hence,
\begin{gather} \label{2nd critical 7}
\begin{split}
0\geq &-\sum_{p,q} \frac{F^{pp,qq}h_{pp1}h_{qq1}}{\kappa_1}-\frac{F^{11}h_{111}^2}{\kappa_{1}^2}+2\sum_{i>m}\frac{F^{ii}h_{ii1}^2}{\kappa_1(\kappa_1-\kappa_i)}\\
&+CN\sum_{i=1}^{n} F^{ii}\kappa_{i}^2+ C\alpha\sum_{i=1}^{n} F^{ii} -\frac{C\beta}{\eta}-C\kappa_1-CN-C.\\
\end{split}
\end{gather}

However, the major difficulty is to deal with the first line in \eqref{2nd critical 6} and for that purpose, we are going to apply some novel ideas from the work of Guan-Ren-Wang \cite{Guan-Ren-Wang}. First, we recall a concavity lemma for the operator $\sigma_k$ and a proof of which can be found in \cite{Lu}.
\begin{lemma}
Let $\kappa=(\kappa_1,\ldots,\kappa_n) \in \Gamma_k$ be ordered as $\kappa_1 \geq \cdots \geq \kappa_n$ and $1 \leq l<k$. For any $\epsilon,\delta,\delta_0 \in (0,1)$, there exists some $\delta'>0$ such that if $\kappa_l \geq \delta \kappa_1$ and $\kappa_{l+1} \leq \delta' \kappa_1$, then we have
\begin{equation}
-\sum_{p,q} \frac{\sigma_{k}^{pp,qq} \xi_p\xi_q}{\sigma_k}+\frac{1}{\sigma_{k}^2}\left(\sum_{i=1}^{n}\sigma_{k}^{ii}\xi_{i}\right)^2 \geq (1-\epsilon)\frac{\xi_{1}^2}{\kappa_{1}^2}-\delta_0\sum_{i>l} \frac{\sigma_{k}^{ii}\xi_{i}}{\kappa_1\sigma_k} \label{Lu's inequality}
\end{equation} for an arbitrary vector $\xi=(\xi_1,\ldots,\xi_n) \in \bR^n$.
\end{lemma}
Then, we proceed by an iteration argument that is very delicate; one could see applications of this new technique in some other settings \cite{Bin-1, Chu, Yang}. We may still let $\varepsilon>0$ denote a very small number without causing any confusion, and let $\delta_{0}=1/2$. Pick an arbitrary $\delta_1 \in (0,1)$, say $\delta_1=1/3$, we would then trivially have $\kappa_1 \geq \delta_1\kappa_1$. Now, by the lemma, there exists some $\delta_2>0$ such that if $\kappa_2 \leq \delta_2\kappa_1$, then the inequality \eqref{Lu's inequality} holds. If we had $\kappa_2 \geq \delta_2 \kappa_1$, then we may continue to pick some $\delta_3$ and see if we would have $\kappa_3 \leq \delta_3 \kappa_1$. The key argument is that this process either halts at some $1 \leq l <k$, or it goes on and we have $\kappa_k > \delta_k \kappa_1$. We now analyze the two cases.

\bigskip

\textbf{Case 1: For each $1 \leq i \leq k$ there exists some $\delta_i>0$ such that $\kappa_i \geq \delta_i\kappa_1$.}

In this case, we immediately have
\begin{align*}
\sigma_k&=\sum_{1 \leq i_1<\cdots<i_k\leq n} \kappa_{i_1}\cdots\kappa_{i_k} \\
&\geq \kappa_1\cdots\kappa_k-C(n,k)\kappa_1\cdots\kappa_{k-1}\cdot K \\
&\geq \kappa_1\cdots\kappa_{k-1}(\kappa_k-CK) \\
&\geq C\delta_2\cdots\delta_{k}\kappa_{1}^{k}
\end{align*} and the desired estimate follows.

\bigskip

\textbf{Case 2: The process halts at some $1 \leq l <k$ and we have \eqref{Lu's inequality}.}

In this case, taking $\xi=(h_{ii1})$ and applying \eqref{differentiate once} and \eqref{approximation}, the inequality \eqref{Lu's inequality} yields
\begin{align*}
&\ -\sum_{p,q} \frac{F^{pp,qq}h_{pp1}h_{qq1}}{\kappa_1}-\frac{F^{11}h_{111}^2}{\kappa_{1}^2}+2\sum_{i>m}\frac{F^{ii}h_{ii1}^2}{\kappa_1(\kappa_1-\kappa_i)}\\
\geq &\ -\frac{(\nabla_1\psi)^2}{\kappa_1\psi}+(1-\varepsilon)\sigma_k\frac{h_{111}^2}{\kappa_{1}^3}-\frac{F^{11}h_{111}^2}{\kappa_{1}^2}+2\sum_{i>m}\frac{F^{ii}h_{ii1}^2}{\kappa_1(\kappa_1-\kappa_i)}-\frac{1}{2}\sum_{i>l}\frac{F^{ii}h_{ii1}^2}{\kappa_{1}^2} \\
\geq &\ -C\kappa_1+(1-\varepsilon)[F^{11}\kappa_1+\sigma_{k}(\kappa|1)]\frac{h_{111}^2}{\kappa_{1}^3}-\frac{F^{11}h_{111}^2}{\kappa_{1}^2}+\sum_{i>\max\{m,l\}}\frac{F^{ii}h_{ii1}^2}{\kappa_{1}^2(\kappa_1-\kappa_i)}(3\kappa_1+\kappa_i)\\
\geq &\ -C\kappa_1-\varepsilon \frac{F^{11}h_{111}^2}{\kappa_{1}^2}+(1-\varepsilon)\sigma_{k}(\kappa|1)\frac{h_{111}^2}{\kappa_{1}^3}.
\end{align*} Therefore,
\begin{gather} \label{2nd critical 8}
\begin{split}
0\geq &\frac{CN}{2}F^{11}\kappa_{1}^2-\varepsilon \frac{F^{11}h_{111}^2}{\kappa_{1}^2}\\
&+ C\alpha\sum_{i=1}^{n} F^{ii} + (1-\varepsilon)\sigma_{k}(\kappa|1)\frac{h_{111}^2}{\kappa_{1}^3}\\
&+\frac{CN}{2}F^{11}\kappa_{1}^2 -\frac{C\beta}{\eta}-C\kappa_1-CN-C.\\
\end{split}
\end{gather}
Still, using the first order critical condition \eqref{1st critical} and the fact that $\kappa_i \geq -K$, we can estimate
\begin{align*}
(1-\varepsilon)\sigma_k(\kappa|1)\frac{h_{111}^2}{\kappa_{1}^3} &\geq -C\kappa_2\cdots\kappa_k \cdot K \cdot \frac{1}{\kappa_1}\cdot \left(\frac{C\beta^2}{\eta^2}+CN^2\kappa_{1}^2+C\alpha^2\right)\\
&\geq -CKN^2\kappa_1\cdots \kappa_k
\end{align*} and 
\begin{align*}
-\varepsilon \frac{F^{11}h_{111}^2}{\kappa_{1}^2}&\geq -\varepsilon F^{11}\left(\frac{C\beta^2}{\eta^2}+CN^2\kappa_{1}^2+C\alpha^2\right) \\
&\geq -\varepsilon CN^2F^{11}\kappa_{1}^2
\end{align*} by assuming $\eta^2\kappa_{1}^2$ is large. Thus,
\[\frac{CN}{2}F^{11}\kappa_{1}^2-\varepsilon \frac{F^{11}h_{111}^2}{\kappa_{1}^2} \geq (CN-\varepsilon CN^2)F^{11}\kappa_{1}^2 \geq 0\] by taking $\varepsilon$ small enough e.g. of order $C/N^2$. On the other hand,
\begin{align*}
C\alpha\sum_{i=1}^{n} F^{ii} + (1-\varepsilon)\sigma_{k}(\kappa|1)\frac{h_{111}^2}{\kappa_{1}^3} &\geq C\alpha\sigma_{k-1}-CKN^2\kappa_1\cdots \kappa_k  \\
&\geq \left(C\alpha-CKN^2 \kappa_k\right)\kappa_1\cdots\kappa_{k-1}.
\end{align*} If $C\alpha-CKN^2\kappa_k \geq 0$, then we are done. Otherwise, taking $\alpha$ to be at least of order $N^3$, we would have $\kappa_k \geq \frac{C\alpha}{KN^2}=CN$ and so
\[\sigma_k \geq \kappa_1\cdots \kappa_k-C\kappa_1\cdots\kappa_{k-1}\cdot K \geq \kappa_1\cdots\kappa_{k-1}(CN-C)\geq CN\kappa_1\] by choosing a large $N$.

Finally, we can remove the first two lines from \eqref{2nd critical 8} and since
\[F^{11}\kappa_{1}^2=\sigma_{k-1}(\kappa|1)\kappa_{1}^2 \geq \frac{k}{n}\sigma_k\kappa_1,\] the estimate follows by choosing $N$ large.

\end{proof}

By an almost identical proof, Theorem \ref{semi-convex} follows. Moreover, we could also provide an alternative proof for Ren-Wang's curvature estimates when $k=n-1$ \cite[Theorem 5]{Ren-Wang-1} and $k=n-2$ \cite[Theorem 7]{Ren-Wang-2}.
\begin{theorem}
Let $k=n-1, n \geq 3$, or $k=n-2, n \geq 5$, and let $\Omega \subseteq \bR^n$ be a convex $(k-1)$-admissible domain with a smooth boundary. Assume $\psi(x,u,Du) \in C^{2}(\overline{\Omega} \times \bR \times \bR^n)$ is positive and $\varphi \in C^2(\overline{\Omega})$ is spacelike. Suppose $u \in C^{4}(\Omega) \cap C^2(\overline{\Omega})$ is a $k$-admissible solution to
\begin{align*}
       \sigma_k[u]&=\psi(x,u, Du)\quad  \text{in $\Omega$},\\
                u&=\varphi \quad\quad\quad\quad\quad     \text{on $\partial\Omega$}.
\end{align*}
Then the maximum principal curvature 
\[\kappa_{\max}(x):=\max_{1 \leq i \leq n} \kappa_{i}(x)\] of its graph satisfies
\[\max_{\Omega} \kappa_{\max}(x) \leq C\left(1+\max_{\partial \Omega} \kappa_{\max}(x)\right)\] for some $C>0$ depending on $n,\norm{u}_{C^1(\overline{\Omega})}, \norm{\psi}_{C^2(\cD)}$ and $\norm{\varphi}_{C^1(\overline{\Omega})}$, where
\[\cD:=\overline{\Omega} \times [\inf_{\Omega} u, \sup_{\Omega} u] \times \bR^n.\]
\end{theorem}
\begin{proof}
In this case, the first line in \eqref{2nd critical 6} can be easily handled because we have the very powerful concavity inequalities \cite[Lemma 2.11]{Bin-1} due to Ren-Wang \cite{Ren-Wang-1, Ren-Wang-2},
\[-\sum_{p,q} \frac{F^{pp,qq}h_{pp1}h_{qq1}}{\kappa_1}-\frac{F^{11}h_{111}^2}{\kappa_{1}^2}+2\sum_{i>m}\frac{F^{ii}h_{ii1}^2}{\kappa_1(\kappa_1-\kappa_i)} \geq -C\frac{(\nabla_1\psi)^2}{\kappa_{1}} \geq -C\kappa_1.\] Moreover, the second line \eqref{the second line} is still non-negative by lemma \ref{negative kappa}:
\[(1-\varepsilon)\kappa_1+(1+\varepsilon)\kappa_i \geq \left(\frac{2k-n}{k}-\frac{n}{k}\varepsilon\right)\kappa_1 \geq 0\] if $0<\varepsilon<(2k-n)/n$ and $2k>n$. Since all the other parts of the calculations remain exactly the same, the proof is complete.
\end{proof}

\begin{remark}
Instead of $\log \kappa_{\max}$, Ren-Wang used $\log \log \sum_{j=1}^{n} e^{\kappa_j}$ in their test function which may cause the calculations to be a lot more involved.
\end{remark}

\section*{Acknowledgements}
Part of this work was carried out while the author was visiting Professor Zhizhang Wang at Fudan University. The author would like to thank Professor Wang for the warm hospitality and for some fruitful discussions on the topic. The visit was financially supported by Professor Man-Chun Lee and the author would like to thank Professor Lee for the generosity. The author would also like to thank Professor Heming Jiao for pointing out a few mistakes in the calculations and for answering inquiries about the paper \cite{Guo-Jiao}.

\section*{Data Availability}
No data was used for the research described in the article.

\bibliography{refs}

\end{document}